\documentclass[preprint,12pt,authoryear]{elsarticle}

\usepackage{amssymb}
\usepackage{amsmath}
\usepackage{amsthm}
\usepackage{tikzit}

\tikzstyle{black dot}=[fill=black, draw=black, shape=circle]
\tikzstyle{red dot}=[fill=red, draw=black, shape=circle]
\tikzstyle{open dot}=[fill=none, draw=black, shape=circle]
\tikzstyle{blue dot}=[fill=blue, draw=black, shape=circle]
\tikzstyle{green dot}=[fill={rgb,255: red,0; green,143; blue,0}, draw=black, shape=circle]
\tikzstyle{pink dot}=[fill={rgb,255: red,128; green,0; blue,128}, draw=black, shape=circle]

\tikzstyle{thin edge}=[-, very thick]
\tikzstyle{directed edge}=[->]
\tikzstyle{double directed}=[<->]
\tikzstyle{red edge}=[-, draw=red, very thick]
\tikzstyle{blue edge}=[-, draw=blue, very thick]
\tikzstyle{pink edge}=[-, draw={rgb,255: red,128; green,0; blue,128}, very thick]
\tikzstyle{green edge}=[-, draw={rgb,255: red,0; green,143; blue,0}, very thick]
\tikzstyle{black dashed}=[-, dashed, very thick]
\tikzstyle{red dashed}=[-, draw=red, dashed, very thick]
\tikzstyle{blue dashed}=[-, draw=blue, dashed, very thick]

\newtheorem{theorem}{Theorem}

\newtheorem{conjecture}[theorem]{Conjecture}
\theoremstyle{definition}

\theoremstyle{definition}

\newtheorem{corollary}{Corollary}[theorem]
\newtheorem{lemma}[theorem]{Lemma}

\journal{Discrete Applied Mathematics}

\begin{document}

\begin{frontmatter}



\title{Minimum Weakly Saturated Graphs and Bootstrap Percolation in General Host Graphs}

\author[1,2]{Roman Vasquez}

\affiliation[1]{organization={Wesleyan College},
    addressline={4760 Forsyth Rd}, 
    city={Macon},
    citysep={},
    postcode={31201}, 
    state={GA},
    country={United States}}

\affiliation[2]{organization={Auburn University},
    city={Auburn},
    citysep={}, 
    postcode={36849}, 
    state={AL},
    country={United States}}

\begin{abstract}

A graph $G$ is weakly $H$-saturated if one can obtain $K_n$ by adding one edge to $G$ at a time, where each additional edge creates at least one new copy of $H$. The minimum number of edges needed for a weakly $H$-saturated graph $G$ of order $n$ is known as the weak saturation number of $H$, written $wsat(n,H)$. A graph $G$ is minimum weakly saturated if $wsat(n,G)=|E(G)|-1$ for some value of $n$.

We explore classes of minimum weakly saturated graphs and their connection to the $H$-bootstrap percolation process, as well as weak saturation in a more general setting than the complete graph.
\end{abstract}



\end{frontmatter}

\section{Introduction}
\label{sec:Introduction}

Assume all graphs are finite and simple. We use the $+$ symbol to refer to the disjoint union of two graphs, let $P_n$ be the path on $n$ vertices, and let $C_n$ be the cycle on $n$ vertices. We will use standard definitions for all graph theory terms that are not otherwise defined.

A graph $G$ of order $n$ is \textit{weakly $H$-saturated} if there exists a nested sequence of graphs $G= G_0 \subset G_1 \subset \cdots \subset G_k = K_n$ such that, for all $i \in [k]$, $G_i$ is obtained from $G_{i-1}$ by adding exactly one edge and $k_H(G_{i-1}) < k_H(G_i)$, where $k_H(G)$ is the number of copies of graph $H$ in $G.$ In other words, $G$ is weakly $H$-saturated if we can obtain $K_n$ by adding one edge from $\overline{G} = K_n\backslash G$ to $G$ at a time, where each additional edge creates at least one new copy of $H$. 

The minimum number of edges needed for a weakly $H$-saturated graph $G$ of order $n$ is the \textit{weak saturation number} of $H$, written $wsat(n,H)$. For a more detailed definition, see \citet{OGWeak}.

\subsection{Bootstrap Percolation on Edges}
\label{sec:Edge_Bootstrap}

There exists a related process known as bootstrap percolation, which can be performed either on the edges or the vertices of a graph $G$, though here we focus on edge bootstrap percolation, referred to from this point on simply as ``bootstrap percolation'' or ``$H$-bootstrap percolation.'' This discrete-time, deterministic process models the spread of an infection from an initial set of edges (referred to as $G_0=E(G)$) to the edges of the complete graph on $|V(G)|$ vertices, according to the presence of subgraphs isomorphic to another graph $H$, which we will call the ``target'' graph. If the entire edge set of the complete graph eventually becomes infected, $G$ (or $G_0$) is said to ``percolate'' or ``$H$-percolate.''

Given a graph $H$ on at most $n$ vertices and initial infected set $G_0$, \textit{$H$-bootstrap percolation} proceeds as follows: at time $t>0$, an edge of $\overline{G_{t-1}}$ becomes infected if it is the only missing edge in an otherwise infected copy of $H$ at time $t-1$, where $G_t$ denotes the infected edges at time $t$. Note that there is no restriction on how many edges can become infected in one round---the primary difference between weak saturation and edge-bootstrap percolation. It is therefore interesting to examine the number of rounds needed to percolate,  written $k=t_f+1$, where $t=t_f$ is the first time step at which $G_t=G_{t+1}$, at which point the process \textit{stabilizes}.

Notice that an order $n$ graph $G$ is weakly $H$-saturated if and only if $G$ $H$-percolates to the complete graph $K_n$ and that it will take at most $|E(\overline{G})|+1$ rounds to percolate. The two processes are identical except in the number of rounds required to stabilize, so we will use the language of weak saturation and bootstrap percolation interchangeably.

Typically the process is performed in the complete graph on $|V(G)|$ vertices, as described, though we will also consider edge bootstrap percolation in a more general graph $X\neq K_n$, in which case the edges not belonging to $E(X)$ are forbidden. Note that we do not require $|V(X)|=|V(G)|$, though this restriction could be accommodated by adding isolates to $G$ as needed. In this general host graph, we say that a subgraph $G$ of $X$ \textit{$H$-percolates to $X$} (or just \textit{percolates to $X$} when $H$ is clear from context) if eventually it stabilizes with $G_t=G_{t+1}=X$.

Note that a graph $G$ will $H$-percolate to $X$ only if every edge of $X$ either belongs to $G_0$ or is contained in a subgraph of $X$ isomorphic to $H$.  

\section{Results}
\label{sec:Results}

The following lemma is frequently useful in proving weak saturation results.

\begin{lemma}
    \label{thm:degree}

    If $G$ and $H$ are graphs, then $\delta(G)+1<\delta(H) \implies G$ does not $H$-percolate to any $K_n$ for $n\geq |V(G)|$.
\end{lemma} 

\begin{proof}
    If $G$ were to $H$-percolate, then at some time step an edge incident to a vertex $v$ with $\deg_G(v)=\delta(G)$ must become infected for the first time. This means that it must be part of a newly-infected copy of $H$, but this is impossible because after adding the new edge $v$ would have degree at most $\deg(v)\leq \delta(G)+1 < \delta(H)$. Therefore, no edge incident to any such $v$ will ever become infected and, as such, $G$ cannot $H$-percolate to the complete graph.
\end{proof}

\begin{corollary}
    \label{cor:degree}
    Suppose $\delta(H)>1$. If a graph $G$ has a vertex of degree 0, then $G$ does not $H$-percolate to any $K_n$ for $n\geq |V(G)|$.
\end{corollary}

Observe that connected graphs can be characterized using weak saturation in the complete graph.

\begin{theorem}
\label{thm:connectivity-k3}

    A graph $G$ of order $n$ is connected if and only if it $K_3$-percolates to $K_n$.
    
\end{theorem}

\begin{proof}
    $(\impliedby)$ Assume to the contrary that $G$ is disconnected. Then at some time step, the first edge between two components will become infected. Say this edge is between vertices $v$ and $u$ which belong to different components. This means that at the prior time step the edge $uv$ must be the only uninfected edge in a copy of $K_3$, but this is impossible because it would mean that $v$ and $u$ have distance 2 despite being in separate components. Thus, $G$ is connected.

    $(\implies)$ Since $G$ is connected, every pair of vertices is a finite distance apart. At each time step, the $K_3$ percolation process will add an infected edge between all vertices of distance 2, therefore decreasing the distance between nonadjacent vertices until they are made neighbors by an infected edge.
\end{proof}

We present the following result where $G \neq H\backslash\{e\}=H-e$; Section \ref{sec:Minimum} will examine minimum weakly saturated graphs with $G = H-e$.

\begin{theorem}
\label{thm:paths-in-cycles}
    For $n\geq m \geq l >2,~ C_m$ will $P_l$-percolate to $K_n$ in at most 3 rounds.
\end{theorem}

\begin{proof}
    In the first time step, all chords of $C_m$ will be infected as they can create a new copy of $P_l$ by following the cycle clockwise (or counterclockwise) from both ends of the chord until the $l$ vertices are present. Notice that, depending on the distance in $C_m$ of the vertices of the chord, one of the trailing ends may need to be longer than the other. If $n=m$, we are done, so assume $n>m$.

    \begin{figure}[h]
    \centering
    \begin{tikzpicture}
	\begin{pgfonlayer}{nodelayer}
		\node [style=black dot] (0) at (-1.25, 3.75) {};
		\node [style=black dot] (1) at (1.25, 3.75) {};
		\node [style=black dot] (2) at (2.75, 2) {};
		\node [style=black dot] (3) at (-2.75, 2) {};
		\node [style=black dot] (4) at (-2.75, 0) {};
		\node [style=black dot] (5) at (2.75, 0) {};
		\node [style=black dot] (6) at (-1.25, -1.75) {};
		\node [style=black dot] (7) at (1.25, -1.75) {};
	\end{pgfonlayer}
	\begin{pgfonlayer}{edgelayer}
		\draw [style=thin edge] (0) to (3);
		\draw [style=red dashed] (3) to (4);
		\draw [style=red dashed] (4) to (6);
		\draw [style=red dashed] (6) to (7);
		\draw [style=thin edge] (7) to (5);
		\draw [style=red dashed] (5) to (2);
		\draw [style=red dashed] (2) to (1);
		\draw [style=red dashed] (1) to (0);
		\draw [style=red dashed] (0) to (7);
	\end{pgfonlayer}
\end{tikzpicture} \quad \quad \quad \begin{tikzpicture}
	\begin{pgfonlayer}{nodelayer}
		\node [style=black dot] (0) at (-1.25, 3.75) {};
		\node [style=black dot] (1) at (1.25, 3.75) {};
		\node [style=black dot] (2) at (2.75, 2) {};
		\node [style=black dot] (3) at (-2.75, 2) {};
		\node [style=black dot] (4) at (-2.75, 0) {};
		\node [style=black dot] (5) at (2.75, 0) {};
		\node [style=black dot] (6) at (-1.25, -1.75) {};
		\node [style=black dot] (7) at (1.25, -1.75) {};
	\end{pgfonlayer}
	\begin{pgfonlayer}{edgelayer}
		\draw [style=thin edge] (0) to (3);
		\draw [style=red dashed] (3) to (4);
		\draw [style=red dashed] (4) to (6);
		\draw [style=red dashed] (6) to (7);
		\draw [style=red dashed] (7) to (5);
		\draw [style=red dashed] (5) to (2);
		\draw [style=thin edge, in=315, out=135] (2) to (1);
		\draw [style=red dashed] (1) to (0);
		\draw [style=red dashed] (0) to (2);
	\end{pgfonlayer}
\end{tikzpicture}
    \caption{The chords of any $C_m$ will be infected in the first round, since they create a new copy of $P_l$. Here, $n=m=l=8$ and the dashed red edges indicate a newly-infected copy of $P_8$ containing the chosen chord.}
    \label{fig:path-in-cycle}
\end{figure}
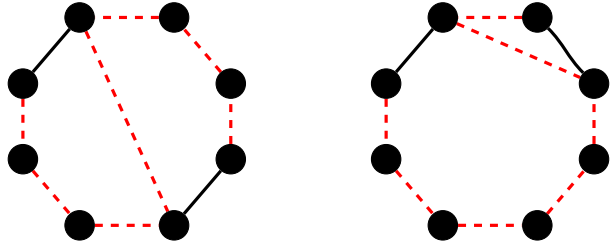

Also in the first time step, every edge from an isolate (a vertex of degree 0) to a vertex in $C_m$ will become infected, as $C_m$ already contains more than enough vertices to complete the copy of $P_l$. If $n=m+1$, we are done, so assume that $n>m+1$.

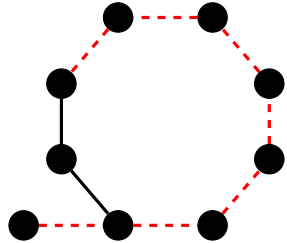
\begin{figure}[h]
    \centering
    \begin{tikzpicture}
	\begin{pgfonlayer}{nodelayer}
		\node [style=black dot] (0) at (-1.25, 3.75) {};
		\node [style=black dot] (1) at (1.25, 3.75) {};
		\node [style=black dot] (2) at (2.75, 2) {};
		\node [style=black dot] (3) at (-2.75, 2) {};
		\node [style=black dot] (4) at (-2.75, 0) {};
		\node [style=black dot] (5) at (2.75, 0) {};
		\node [style=black dot] (6) at (-1.25, -1.75) {};
		\node [style=black dot] (7) at (1.25, -1.75) {};
		\node [style=black dot] (8) at (-3.75, -1.75) {};
	\end{pgfonlayer}
	\begin{pgfonlayer}{edgelayer}
		\draw [style=red dashed] (0) to (3);
		\draw [style=thin edge] (3) to (4);
		\draw [style=thin edge] (4) to (6);
		\draw [style=red dashed] (6) to (7);
		\draw [style=red dashed] (7) to (5);
		\draw [style=red dashed] (5) to (2);
		\draw [style=red dashed] (2) to (1);
		\draw [style=red dashed] (1) to (0);
		\draw [style=red dashed] (8) to (6);
	\end{pgfonlayer}
\end{tikzpicture}
    \caption{All edges between the cycle and the $n-m$ isolates become infected at $t=1$. Here, $n=9, ~m=l=8$ and the dashed red edges indicate a newly-infected copy of $P_8$ containing the chosen edge.}
    \label{fig:path-in-cycle-2}
\end{figure}


The only remaining uninfected edges are between the $n-m$ vertices which did not belong to the initial $C_m$. These will become infected at the next time step, $t=2$, using any of the already-infected copies of $P_{l-1}$ containing either of the endpoints of the new edge. By this point, all edges in $K_n$ have been infected, meaning that the process stabilizes no later than $t=2$ and $k\leq 3$.

\end{proof}

\subsection{Minimum Weakly Saturated Graphs}
\label{sec:Minimum}

First note that $wsat(n,H) \geq |E(H)|-1$. A graph $H$ is \textit{minimum weakly saturated} if there exists some value of $n\geq |V(H)|$ such that $wsat(n,H) = |E(H)|-1$. Alternatively, $H$ is minimum weakly saturated if there exists some $n\geq |V(H)|$ such that $H-e$ will $H$-percolate to $K_n$ for some $e\in E(H)$. $H$ is \textit{totally minimum weakly saturated} if $wsat(n,H) = |E(H)|-1$ for all values of $n\geq |V(H)|$.

This differs slightly from previous use of the term ``minimum weakly saturated'' in the literature, which is sometimes defined the same way \textit{totally} minimum weakly saturated is used here or which may only consider cases in which $G$ spans the host graph; see \citet{WeakSatSparse, Joins}.

A minimum weakly saturated graph $H$ will $H$-percolate to $K_n$ for some $n\geq |V(H)|$; starting with $H-e$ is not necessary, as the edge $e$ will always be infected in the first round. Thus, we refer to a minimum weakly saturated graph $H$ as \textit{self-percolating} (up to $K_n$ if the host graph $X$ is not specified).

We examine graphs $H$ which are self-percolating and determine the number of rounds needed to percolate, beginning with the following result for complete graphs, presented without proof.

\subsubsection{Complete Graphs}
\label{sec:complete}

\begin{theorem}
\label{thm:complete-graphs}
    $K_p$ is self-percolating for $p>2$. $K_p-e$ will $K_p$-percolate to $K_n$ only in the case where $n=p$, in which case it will percolate in $k=2$ rounds.
\end{theorem}

\subsubsection{Paths}
\label{sec:paths}

We will present a result concerning self-percolation of paths, but we first need a useful lemma.

\begin{lemma}
\label{lem:containment}
    Consider $H$-bootstrap percolation in a host graph $X$ of two order $n$ graphs $G$ and $G'$.
    If $G_{t} \subseteq G'_{s}$ for some time steps $t$ and $s$, then $G$ $H$-percolates $\implies G'$ $H$-percolates. Moreover, if $G$ percolates in $k\geq t+1$ steps, then $G'$ percolates in at most $k-t+s$ time steps.
\end{lemma}

\begin{proof} 

Starting with $G_{t} \subseteq G'_{s}$, we know that at each new time step every newly-infected edge of $X$ to be added to $G_{t}$ in creating $G_{t+1}$ is either already in $G'_{s}$ or will be added to $G'_{s}$ to create $G'_{s+1}$. In either case, $G_{t+1} \subseteq G'_{s+1}$. Iterating, we see that $G_{t+z}\subseteq G'_{s+z}$ for all $z \in \mathbb{Z}^+$. Therefore $X=G_{t+k-t}\subseteq G'_{s+k-t} \subseteq X$, so $G'$ percolates in at most $k-t+s$ rounds.

\end{proof}

Now to deal with paths.

\begin{theorem}
\label{thm:paths}
    $P_l$ (for $l\geq 3$) is totally minimum weakly saturated.

    Moreover, $P_l$ will self-percolate to $K_n$ in $k\leq3$ rounds for any $n\geq l$.
\end{theorem}

\begin{proof}

\textbf{Case 1:} $n=l$

At the first time step, the edge connecting the two ends of the original $P_l$ will be infected, creating a copy of $C_l$. According to Lemma \ref{lem:containment} and Theorem \ref{thm:paths-in-cycles}, all chords of $C_l$ become infected in the next time step, using the same clockwise strategy seen in the proof of Theorem \ref{thm:paths-in-cycles}.

\textbf{Case 2:} $n>l$

Since $n> l$, there is at least one isolated vertex. In the first time step, edges from the ends of $P_l$ to the isolates will be infected. Some additional edges are infected, though they are not needed for the next round. This results in the creation of an infected $C_{l+1}$ for each former isolate, so again we apply Lemma \ref{lem:containment} and use the same process as in Theorem \ref{thm:paths-in-cycles} to infect every chord. Since every vertex belongs to at least one infected $C_{l+1}$ at time $t=1$, the process will stabilize at time $t=2$ and therefore $k= 3$.

 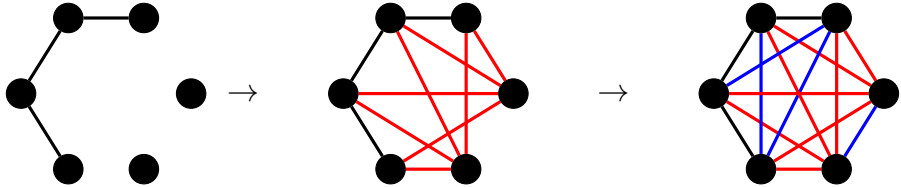
\begin{figure}[h]
        \centering
        \begin{tikzpicture}
	\begin{pgfonlayer}{nodelayer}
		\node [style=black dot] (0) at (-1, 2) {};
		\node [style=black dot] (1) at (1, 2) {};
		\node [style=black dot] (2) at (2.25, 0) {};
		\node [style=black dot] (4) at (-1, -2) {};
		\node [style=black dot] (5) at (-2.25, 0) {};
		\node [style=black dot] (6) at (1, -2) {};
	\end{pgfonlayer}
	\begin{pgfonlayer}{edgelayer}
		\draw [style=thin edge] (4) to (5);
		\draw [style=thin edge] (5) to (0);
		\draw [style=thin edge] (0) to (1);
	\end{pgfonlayer}
\end{tikzpicture}
        $\rightarrow$ \hspace{.5cm}
        \begin{tikzpicture}
	\begin{pgfonlayer}{nodelayer}
		\node [style=black dot] (0) at (-1, 2) {};
		\node [style=black dot] (1) at (1, 2) {};
		\node [style=black dot] (2) at (2.25, 0) {};
		\node [style=black dot] (4) at (-1, -2) {};
		\node [style=black dot] (5) at (-2.25, 0) {};
		\node [style=black dot] (6) at (1, -2) {};
	\end{pgfonlayer}
	\begin{pgfonlayer}{edgelayer}
		\draw [style=thin edge] (4) to (5);
		\draw [style=thin edge] (5) to (0);
		\draw [style=thin edge] (0) to (1);
		\draw [style=red edge] (1) to (2);
		\draw [style=red edge] (1) to (6);
		\draw [style=red edge] (4) to (6);
		\draw [style=red edge] (4) to (2);
		\draw [style=red edge] (5) to (6);
		\draw [style=red edge] (5) to (2);
		\draw [style=red edge] (0) to (6);
		\draw [style=red edge] (0) to (2);
	\end{pgfonlayer}
\end{tikzpicture} \hspace{.5cm}
        $\rightarrow$ \hspace{.5cm}
        \begin{tikzpicture}
	\begin{pgfonlayer}{nodelayer}
		\node [style=black dot] (0) at (-1, 2) {};
		\node [style=black dot] (1) at (1, 2) {};
		\node [style=black dot] (2) at (2.25, 0) {};
		\node [style=black dot] (4) at (-1, -2) {};
		\node [style=black dot] (5) at (-2.25, 0) {};
		\node [style=black dot] (6) at (1, -2) {};
	\end{pgfonlayer}
	\begin{pgfonlayer}{edgelayer}
		\draw [style=thin edge] (4) to (5);
		\draw [style=thin edge] (5) to (0);
		\draw [style=thin edge] (0) to (1);
		\draw [style=red edge] (1) to (2);
		\draw [style=red edge] (1) to (6);
		\draw [style=red edge] (4) to (6);
		\draw [style=red edge] (4) to (2);
		\draw [style=blue edge] (2) to (6);
		\draw [style=red edge] (5) to (6);
		\draw [style=red edge] (0) to (2);
		\draw [style=blue edge] (4) to (1);
		\draw [style=red edge] (0) to (6);
		\draw [style=blue edge] (0) to (4);
		\draw [style=red edge] (5) to (2);
		\draw [style=blue edge] (5) to (1);
	\end{pgfonlayer}
\end{tikzpicture}
        \caption{Case 2 in $P_l$ self-percolation, with $l=4$ and $n=6$.}
        \label{fig:path-self-example}
    \end{figure}

\end{proof}

\subsubsection{Cycles}
\label{sec:cycles}

Cycles, however, are not self-percolating.

\begin{theorem}
    $C_m$ will not self-percolate to any $K_n$ for $n\geq m$.
\end{theorem}

\begin{proof}
    The only uninfected edges are chords of $C_m$ and edges with at least one endpoint not on the cycle. None of these can be infected: chords will not be infected because they can only create smaller cycles and Theorem \ref{thm:degree} tells us that if $m<m$, it will never percolate.
\end{proof}

\subsubsection{Kayak Paddle Graphs}
\label{sec:kayak}

Graphs containing disjoint cycles may still self-percolate, as seen in the following result concerning kayak paddle graphs. The kayak paddle graph $KP(k,m,l)$ is the graph obtained by connecting the graphs $C_k$ and $C_m$ by a path with $l$ edges. We show that kayak paddle graphs of the form $KP(k,m,1)$ are minimum weakly saturated, but not totally.

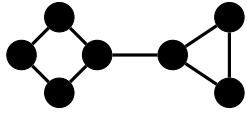
\begin{figure}[h]
    \centering
    \begin{tikzpicture}
	\begin{pgfonlayer}{nodelayer}
		\node [style=black dot] (0) at (-1, 0) {};
		\node [style=black dot] (1) at (-2, 1) {};
		\node [style=black dot] (2) at (-3, 0) {};
		\node [style=black dot] (3) at (-2, -1) {};
		\node [style=black dot] (4) at (1, 0) {};
		\node [style=black dot] (5) at (2.5, 1) {};
		\node [style=black dot] (6) at (2.5, -1) {};
	\end{pgfonlayer}
	\begin{pgfonlayer}{edgelayer}
		\draw [style=thin edge] (0) to (4);
		\draw [style=thin edge] (4) to (5);
		\draw [style=thin edge] (5) to (6);
		\draw [style=thin edge] (6) to (4);
		\draw [style=thin edge] (0) to (1);
		\draw [style=thin edge] (1) to (2);
		\draw [style=thin edge] (2) to (3);
		\draw [style=thin edge] (3) to (0);
	\end{pgfonlayer}
\end{tikzpicture}
    \caption{The kayak paddle $KP(4,3,1)$}
    \label{fig:kayak-paddle-example}
\end{figure}

\begin{theorem}
    $KP(k,m,1)$ will self-percolate to $K_{k+m}$ and no other complete graph, and it will do so in at most $k=3$ rounds.
\end{theorem}


\begin{proof}
   Let the host graph be $K_n$ and note that $|V(H)|=k+m$. By Lemma \ref{thm:degree}, it must be the case that $n=k+m$ if we hope to percolate to $K_n$.

Both cycles begin fully infected, so every edge between the cycles will be infected in the first time step, since any of these edges could serve as the bridge. All of the remaining uninfected edges are chords of cycles, which will become infected on the next time step.

Consider a chord $e'$ of $C_m$; we will find an infected copy of $H=KP(k,m,1)$ in which $e'$ is the bridge between an infected $m$-cycle $A$ and an infected $k$-cycle $B$. Call the originally infected bridge $e$.

Because both endpoints of $e'$ belong to the already-infected $C_m$, we have an already-infected path on $m-1$ vertices which contains an endpoint of $e'$, say $x$. We use these edges, along with two others, to form $A$. Since all edges from $C_m$ to $C_k$ were infected in the previous time step, there is an infected edge from each of the endpoints of this path to any arbitrarily chosen vertex $x'$ in $C_k$, completing the infected $m$-cycle $A$, as seen in Figure \ref{fig:paddle1}, along with the infected bridge $e'$.

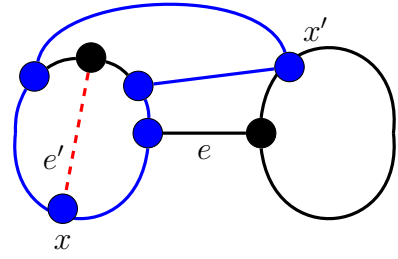
\begin{figure}[h]
    \centering
    \begin{tikzpicture}
	\begin{pgfonlayer}{nodelayer}
		\node [style=blue dot] (0) at (-1.5, 0) {};
		\node [style=black dot] (1) at (1.5, 0) {};
		\node [style=none] (2) at (-5, 0) {};
		\node [style=none] (3) at (5, 0) {};
		\node [style=black dot] (4) at (-3, 2) {};
		\node [style=blue dot] (5) at (-3.75, -2) [label=below:{$x$}] {};
		\node [style=blue dot] (7) at (-1.75, 1.25) {};
		\node [style=blue dot] (8) at (2.25, 1.75) [label=80:{$x'$}] {};
		\node [style=blue dot] (9) at (-4.5, 1.5) {};
	\end{pgfonlayer}
	\begin{pgfonlayer}{edgelayer}
		\draw [style=thin edge] node[below] {$e$} (0) to (1);
		\draw [style=blue edge, bend right=90, looseness=2.00] (2.center) to (0);
		\draw [style=thin edge, bend left=90, looseness=2.00] (1) to (3.center);
		\draw [style=thin edge, bend right=90, looseness=2.00] (1) to (3.center);
		\draw [style=red dashed] (4) to  node[below left] {$e'$} (5);
		\draw [style=blue edge] (7) to (8);
		\draw [style=blue edge, bend right=15] (9) to (2.center);
		\draw [style=blue edge, bend right=15, looseness=0.75] (0) to (7);
		\draw [style=thin edge, bend right=15] (7) to (4);
		\draw [style=thin edge, bend right=15] (4) to (9);
		\draw [style=blue edge, bend left=285, looseness=0.75] (8) to (9);
	\end{pgfonlayer}
\end{tikzpicture}
    \caption{Infecting the chord $e'$ (dashed red line). The cycle on the left is the original $C_m$ and the cycle on the right is the original $C_k$. The infected copy of the $m$-cycle $A$ is in blue.}
    \label{fig:paddle1}
\end{figure}

To complete the newly-infected copy of $H$ containing chord $e'$, all we need is for the other endpoint of $e'$ (say $y$) to belong to the same already-infected cycle as the remaining $k-1$ vertices.

Because we've only used one vertex $x'$ in $C_k$, there is still an infected path $P_{k-1}$ on the other $k-1$ vertices of $C_k$. Since $y$ belongs to the original $C_m$ and every edge from that $C_m$ to the original $C_k$ is infected, use the edges between $y$ and each end of $P_{k-1}$ to finish the infected $k$-cycle $B$.

\begin{figure}[h]
    \centering
    \begin{tikzpicture}
	\begin{pgfonlayer}{nodelayer}
		\node [style=blue dot] (0) at (-1.5, 0) {};
		\node [style=green dot] (1) at (1.5, 0) {};
		\node [style=none] (2) at (-5, 0) {};
		\node [style=none] (3) at (5, 0) {};
		\node [style=green dot] (4) at (-3.25, 2) [label, pin=155:{$y$}] {};
		\node [style=blue dot] (5) at (-3.75, -2) [label=below:{$x$}] {};
		\node [style=blue dot] (7) at (-2, 1.25) {};
		\node [style=blue dot] (8) at (1.75, 1.5) [label, pin=85:{$x'$}] {};
		\node [style=blue dot] (9) at (-4.5, 1.5) {};
		\node [style=green dot] (10) at (3.5, 2) {};
	\end{pgfonlayer}
	\begin{pgfonlayer}{edgelayer}
		\draw [style=thin edge] node[below] {$e$} (0) to (1);
		\draw [style=blue edge, bend right=90, looseness=2.00] (2.center) to (0);
		\draw [style=green edge, bend right=90, looseness=2.00] (1) to (3.center);
		\draw [style=red dashed] (4) to  node[below left] {$e'$} (5);
		\draw [style=blue edge] (7) to (8);
		\draw [style=blue edge, bend right=15] (9) to (2.center);
		\draw [style=blue edge, bend right=15] (0) to (7);
		\draw [style=thin edge, bend right, looseness=0.75] (7) to (4);
		\draw [style=thin edge, bend right=15] (4) to (9);
		\draw [style=blue edge, bend left=285, looseness=0.75] (8) to (9);
		\draw [style=green edge, bend right, looseness=0.75] (10) to (4);
		\draw [style=green edge, in=165, out=0] (4) to (1);
		\draw [style=green edge, bend left=45, looseness=0.75] (10) to (3.center);
		\draw [style=thin edge, bend right=45, looseness=1.25] (10) to (1);
	\end{pgfonlayer}
\end{tikzpicture}
    \caption{Finding an infected $k$-cycle $B$ in green.}
    \label{fig:paddle2}
\end{figure}
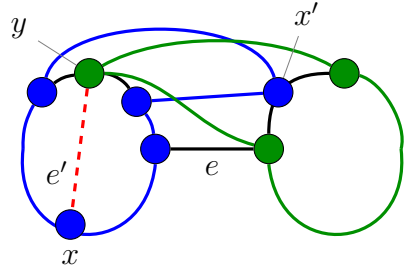

Chords of the existing $k$-cycle are infected similarly.

Thus, the remaining edges are infected and $G$ will $H$-percolate in a maximum of 2 time steps or $k=3$ rounds.

\end{proof}

\subsubsection{Double Star Graphs}
\label{sec:double-star}

Similarly, some double star graphs are minimum weakly saturated, but not totally. Define the double star $S(n,m)$ as the graph obtained by joining an $n$-star ($K_{1,n}$) and an $m$-star ($K_{1,m}$) together with a single edge between their center vertices.

\begin{theorem}
\label{thm:double-star}

    Suppose that $n$ and $m$ are integers with $1\leq n \leq m$. If $S(n,m)\not\in \{S(1,1), S(1,2)\}$, then $S(n,m)$ self-percolates to $K_{m+n+2+q}$ if and only if $n\leq2$ and $q>0$. The $S(n,m)$-percolation is accomplished in $k\leq 4$ rounds. 
\end{theorem}

\begin{proof}
    $(\impliedby)$ Start with $G=H\cup V=S(n,m)+qK_1$ where $V$ is a nonempty set of $q$ vertices and $n \leq 2$. We will show that this graph $S(m,n)$-percolates. Let $v_n$ be the center vertex of the $n$-star and $v_m$ the center of the $m$-star.

    In the first round, the edges between $v_i$ ($i \in \{n,m\}$) and the $q$ isolated vertices will be infected, as any of these edges can replace an edge from their respective star.

    \begin{figure}[h]
    \centering
    \begin{tikzpicture}
	\begin{pgfonlayer}{nodelayer}
		\node [style=black dot] (0) at (-3, 1) {};
		\node [style=black dot] (1) at (-3, -1) {};
		\node [style=black dot] (2) at (-1, 3) {};
		\node [style=black dot] (3) at (1, 3) {};
		\node [style=black dot] (4) at (3, 1) {};
		\node [style=black dot] (5) at (3, -1) {};
		\node [style=black dot] (6) at (-1, -3) {};
		\node [style=black dot] (7) at (1, -3) {};
		\node [style=black dot] (8) at (-3, -1) {};
	\end{pgfonlayer}
	\begin{pgfonlayer}{edgelayer}
		\draw [style=thin edge] (0) to (8);
		\draw [style=thin edge] (0) to (2);
		\draw [style=thin edge] (0) to (3);
		\draw [style=thin edge] (0) to (4);
		\draw [style=thin edge] (8) to (6);
		\draw [style=thin edge] (8) to (7);
		\draw [style=red dashed] (8) to (5);
		\draw [style=red dashed] (0) to (5);
	\end{pgfonlayer}
\end{tikzpicture}
    \caption{An example of $G_1$ where $H=S(2,3)$ and $G=S(2,3)+1K_1$. Edges of $G$ are shown in black, edges in $G_1\backslash G$ are dashed and red.}
    \label{fig:double-star-1}
\end{figure}
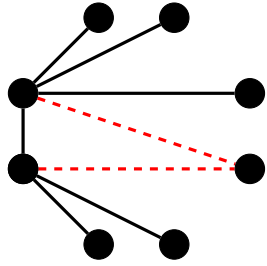

    During the next time step, the remaining edges incident to the $v_i$ will become infected. To see why, say that $i \neq j \in \{n, m \}$ and consider the infected $(j+q)$-star with $v_j$ at its center. In the previous round, $v_j$ gained $q>0$ new neighbors, meaning that any arbitrary choice of leaf $x$ can be excluded and still leave enough leaves for the necessary $j$-star. Then the edge $xv_i$, along with $i-1$ of the edges from the original $i$-star, is used to form the newly-infected $i$-star.

    Additionally, edges between the $q$ vertices in $V$ will become infected at time $t=2$. If $q=1$, there are no such edges, so assume $q \geq 2$ and let $v,w \in V$. We must find an (otherwise) infected copy of $S(n,m)$ containing edge $vw$.
    
    One of the two vertices, say $v$, will serve as the copy of $v_n$, while the other will be a leaf of the $n$-star. Either $vv_n$ or $vv_m$ will be the copy of $v_nv_m$, without loss of generality say $vv_m$, and the other (in this case $vv_n$) will be the other leaf of the $n$-star if $n=2$. If $n=1$ this edge is not necessary and the vertex $v_n$ will not belong to the newly-infected copy of the double star.

    With $v_m$ serving as the copy of $v_m$, we need $m$ leaves; only $v$ and $w$ are used already in this copy, so there are $(m+n+q-2)\geq m$ options left (all of the edges from those vertices to $v_m$ were already infected), as required. This completes the copy of $S(n,m)$.

    \begin{figure}[h]
    \centering
    \begin{tikzpicture}
	\begin{pgfonlayer}{nodelayer}
		\node [style=black dot] (0) at (-3, 1) {};
		\node [style=black dot] (1) at (-3, -1) {};
		\node [style=black dot] (2) at (-1, 3) {};
		\node [style=black dot] (3) at (1, 3) {};
		\node [style=black dot] (4) at (3, 1) {};
		\node [style=black dot] (5) at (3, -1) {};
		\node [style=black dot] (6) at (-1, -3) {};
		\node [style=black dot] (7) at (1, -3) {};
		\node [style=black dot] (8) at (-3, -1) {};
	\end{pgfonlayer}
	\begin{pgfonlayer}{edgelayer}
		\draw [style=thin edge] (0) to (8);
		\draw [style=thin edge] (0) to (2);
		\draw [style=thin edge] (0) to (3);
		\draw [style=thin edge] (0) to (4);
		\draw [style=thin edge] (8) to (6);
		\draw [style=thin edge] (8) to (7);
		\draw [style=red edge] (8) to (5);
		\draw [style=red edge] (0) to (5);
		\draw [style=blue dashed] (4) to (8);
		\draw [style=blue dashed] (3) to (8);
		\draw [style=blue dashed] (2) to (8);
		\draw [style=blue dashed] (7) to (0);
		\draw [style=blue dashed] (6) to (0);
	\end{pgfonlayer}
\end{tikzpicture}
    \caption{The black edges belong to $G$ alone, the red edges belong to $G_1\backslash G$, and the dashed blue edges belong to $G_2\backslash G_1$. Because $q=1$, there was only one type of edge infected at $t=2$.}
    \label{fig:double-star-2}
\end{figure}
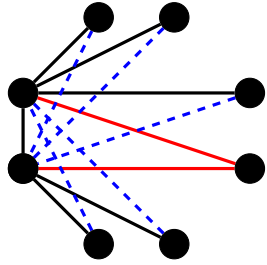

At time $t=3$, the remaining edges between the former leaves are now infected in the same manner as the edges between vertices of $V$ were infected in the previous round.

$(\implies)$ \textbf{Case 1:} $n>2$

The first time step must proceed exactly the same as above, and the edges from each $v_i$ (for $i\in \{n,m\}$) to each of the original $(n+m)$ leaves is infected in the same manner at time $t=2$. However, no other edges will ever become infected.

The only remaining uninfected edges are of the form $x_1x_2$, where $x_1,x_2 \in V(G)\backslash\{v_n,v_m\}$, and each of these vertices has degree 2 in $G_2$. Thus, infecting one of these edges would produce a pair of adjacent vertices of degree at most 3, but our target graph $S(m,n)$ contains no pair of adjacent vertices of degree 3 or less, so there is no way that the new edge produces a new copy of the double star.

\textbf{Case 2:} $q=0$, i.e. $V= \varnothing$

No edges will be infected after the initial infected set $G_0$ is selected.

The missing edges are one of three types: edges from $v_m$ to leaves of the $n$-star (leaves which we will denote $n_i$ with $i\in [n]$), edges from $v_n$ to leaves of the $m$-star (denoted $m_j$ with $j\in [m]$), and edges between leaves of the two stars. Neither of the first two can be infected: say we wish to infect $v_mn_i$. Then the $n$-star will have $n-1$ choices of leaves with which to form the new $n$-star. Similarly, an edge of the form $v_nm_j$ cannot be infected at this point.

Edges between leaves will create a pair of adjacent vertices of degree 2---the same problem seen in case 1. There are no edges that can be infected to create a new infected copy of $S(n,m)$.

Thus,  $S(n,m)$ is self-percolating if and only if $n \leq 2$ and only in a complete graph with more than $n+m+2$ vertices.
    
\end{proof}

There is a class of graphs known as double broom graphs, closely related to double star graphs. Let $B(n,m,k)$ be the graph consisting of an $n$-star and an $m$-star connected by a path with $k$ edges. Notice that all double stars $S(n,m)$ are also double brooms of the form $B(n,m,1)$.

While we have not confirmed any results regarding minimum weak saturation in double brooms, we suspect the following:

\begin{conjecture}
    If a double broom graph is minimum weakly saturated, then it has the form $B(n,m,2k+1)$, where $n,m\geq 1, ~k\geq 0$.
    \label{con:double-broom}
\end{conjecture}

\subsubsection{Complete Multipartite Graphs}
\label{sec:multipartite}

Existing results focus on the asymptotic behavior of complete multipartite hypergraphs using algebraic methods \cite{MultiPartite}. Here, we give combinatorial proofs regarding minimum weakly saturated complete multipartite graphs.

It is simple to see that complete bipartite graphs are not self-percolating (with the exception of $K_{1,2}$ or $K_{1,1}$). 

The more general complete $r$-partite graph is more complicated, however, and is self-percolating in some cases.

\begin{theorem}
\label{thm:multipartite}

Let $H=K_{n_1,n_2,...,n_r}$ be a complete $r$-partite graph (where $r> 2$). The graph $H$ will self-percolate to $K_n$ if and only if the following conditions hold: 

\begin{enumerate}
    \item $\Sigma_{i=1}^r n_i=n$, \\

    \item $1\leq n_i \leq 2$ for $i \in [r]$, and \\

    \item $n_i=1$ for at least one $i \in [r]$
\end{enumerate}

in which case the graph percolates in $k=2$ time steps.

\end{theorem}

\begin{proof}
    $(\impliedby)$ Consider edge $xy \notin E(H)$ for infection. We will show that if conditions 1-3 hold, then there exists a copy $H'$ of $H$ in $K_n$ with $xy \in E(H')$ such that all other edges of $H'$ also belong to $E(H)$.
    
    Since $H$ is a complete $r$-partite graph, vertices $x$ and $y$ must belong to the same partite set, say $I$, and $|I|=2$. Then $x$ and $y$ are each adjacent to all other vertices in $H$. Additionally, there exists a partite set $J$ with $|J|=1$, say that $J = \{z\}$. The edges $xz$ and $yz$ must belong to $H$ and are already infected.

    Let $x \in I'$ and $y \in J'$, where $I'$ and $J'$ are the copies of $I$ and $J$ in $H'$, respectively. The vertex $z$ will join $x$ in $I'$; notice that $z$ was already adjacent to every other vertex in $H$ and therefore every edge between $I'$ and $H\backslash I'$ besides $xy$ is already infected. We have already confirmed that the edges incident to $y \in J'$ are infected. Leaving the rest of the edges in $H$ alone, we have that $xy$ is the only uninfected edge necessary to form $H'$.

    $(\implies)$ We show that failing to meet one of the three listed conditions guarantees that $H$ is not self-percolating.

    \textbf{Case 1:} $n \neq \Sigma_{i=1}^r n_i$.

    Because $r>2$, then there are no vertices of degree at most one and so by Corollary \ref{cor:degree}, $H$ cannot be self-percolating.

    \textbf{Case 2:} There exists $i \in [r]$ such that $n_i \geq 3$.

    Let $x,~y,$ and $z$ be distinct vertices in the same partite set and assume for the sake of contradiction that $H$ self-percolates. Then eventually all 3 of the edges $xy,~xz,$ and $yz$ will become infected; without loss of generality say that the edge $xy$ is the first of these to become infected. Then $xy$ must belong to an otherwise already infected copy $H'$ of $H$ such that $xz,~yz \notin E(H')$. 

    Since $H'$ is a complete multipartite graph, $x$ and $z$ must belong to the same partite set in $H'$ and so must $y$ and $z$. Thus, $x$ and $y$ must also be in the same partite set in $H'$ which is impossible. Hence, $H$ is not self-percolating.

    \textbf{Case 3:} $n_i \geq 2$ for all $i \in [r]$.

    By Case 2, we can assume  $n_i = 2$ for all $i \in [r]$. Again, assume for contradiction that $H$ self-percolates.

    Consider the first new edge to become infected, say $xy$ with $x$ and $y$ in partite set $A$. Say that $x$ belongs to the copy $A'$ of $A$ in the newly infected copy $H'$ and $y$ belongs to some other partite set $B'$ in $H'$.

    Because $|A|=|A'|=2$, there is some other vertex $z \notin A$ such that $z$ joins $x$ in $A'$ in this copy of $H'$, replacing $y$. This means that $z$ must be adjacent to all vertices besides $x$ in $H'$---including the vertex that is in the same partite set as $z$ in the original $H$. This is not possible and, therefore, $H$ must not be self-percolating
    
\end{proof}

\section{Conclusion}

Extremal problems are a relatively new area of study in the field of bootstrap percolation; this is especially true for edge bootstrap percolation. We present some initial findings and a conjecture regarding self-percolating graphs. There still exist a wide variety of unexplored classes of graphs which may be self-percolating---including the double broom graphs that are the subject of Conjecture \ref{con:double-broom}.

Similarly, there exists essentially no mention of edge bootstrap percolation or weak saturation in host graphs other than the complete graph. We present only a few results that generalize to host graphs $X$ but invite further attention to this subject.

\section{Acknowledgments}

The author would like to thank Dr. Peter Johnson at Auburn University for his swift and honest feedback.

The author declares that they have no known competing financial interests or personal relationships that could have appeared to influence the work reported in this paper. This research did not receive any specific grant from funding agencies in the public, commercial, or not-for-profit sectors.


\bibliographystyle{cas-model2-names}

\bibliography{Bibliography.bib}

\end{document}